\documentclass[11pt,twoside]{article}

\usepackage{mathrsfs,amsfonts,amsmath,amssymb}
\usepackage{latexsym}
\usepackage{graphicx}
\newtheorem{satz}{Theorem}

\newtheorem{theorem}[satz]{Theorem}
\newtheorem{conjecture}[satz]{Conjecture}
\newtheorem{lemma}[satz]{Lemma}

\newtheorem{corollary}[satz]{Corollary}

\renewcommand{\thefootnote}{}

\def\T{\mathsf{T}}

\def\F{\mathbb {F}}

\def\E{\mathsf{E}}

\def\P{{\cal P}}

\def\({\big (}
\def\){\big )}

\def\le{\leqslant}
\def\ge{\geqslant}
\def\_phi{\varphi}

\def\T{{\sf T}}
\def\E{{\sf E}}

\title{Stronger sum--product inequalities for small sets}

\author{M. Rudnev, G. Shakan, I. D. Shkredov}

\begin{document}

\pagenumbering{arabic} \setcounter{page}{1}

\maketitle

\begin{abstract}
Let $F$ be a field and  a finite $A\subset F$ be sufficiently small in terms of the characteristic $p$ of $F$ if $p>0$. 

We strengthen the ``threshold'' sum--product inequality  $$|AA|^3 |A\pm A|^2 \gg |A|^6\,,\;\;\;\;\mbox{hence} \;\; \;\;|AA|+|A+A|\gg |A|^{1+\frac{1}{5}},$$
due to Roche--Newton, Rudnev and Shkredov, to
 $$|AA|^5 |A\pm A|^4 \gg |A|^{11-o(1)}\,,\;\;\;\;\mbox{hence} \;\; \;\;|AA|+|A\pm A|\gg |A|^{1+\frac{2}{9}-o(1)},$$
as well as 
$$
|AA|^{36}|A-A|^{24} \gg |A|^{73-o(1)}.
$$
The latter inequality is ``threshold--breaking'', for it shows for $\epsilon>0$, one has
$$|AA| \le |A|^{1+\epsilon}\;\;\;\Rightarrow\;\;\;
|A-A|\gg  
|A|^{\frac{3}{2}+c(\epsilon)},$$
with $c(\epsilon)>0$ if $\epsilon$ is sufficiently small.

This implies that regardless of $\epsilon$,  $$|AA-AA|\gg |A|^{\frac{3}{2}+\frac{1}{56}-o(1)}\,.$$
\end{abstract}

\section{Introduction and results} Let $F$ be a field with the multiplicative group $F^*$. Throughout we assume that $F$ has characteristic $p>0$, the most important case being $F = \mathbb{F}_p$ for large $p$. If $p=0$, constraints in terms of $p$ appearing throughout should be disregarded.

All the sets $A,B$, etc. considered are finite, of cardinality $|\cdot|$; one defines the sumset   via  
$$A+B :=\{a+b: \,a \in A, b\in B\}$$
and similarly the difference, product set, as well as polynomial expressions like $AA-AA$ used herein. If $B=\{b\}$ we just write $A+b$ for $A+\{b\}.$ In contrast, the notation $A^n$ denotes the $n$--fold Cartesian product of $A$ with itself.

The study of the so--called sum--product phenomenon originated in the paper  \cite{ES} by  Erd\H{o}s and Szemer\'edi, who conjectured the following.

\begin{conjecture}[Sum--product conjecture]\label{esconj}\cite{ES} For $\delta < 1$ and any sufficiently large $A \subset \mathbb{Z}$, one has
 $$|A+A| + |A A| \geq |A|^{1 + \delta}.$$
\end{conjecture}

 Elekes in his foundational paper \cite{El} observed that if the question of Erd\H os and Szemer\'edi is asked over a field rather than a ring, then one can use incidence geometry and make good progress on it. Fields, beginning with reals, where Elekes fetched the Szemer\'edi-Trotter theorem as a powerful tool, have become the structure of choice for variants of Conjecture \ref{esconj} ever since.

The study of asymptotic sum--product estimates in fields of positive characteristic began in the prime residue field $\F_p$ setting  by Bourgain, Katz and Tao \cite{BKT}  where the first qualitative result was established. It was made quantitative by Garaev \cite{Ga}, whose paper was followed by a body of incremental improvements. The new wave of quantitative results was initiated in  \cite{Ru} and \cite{RRS}, based on the point-plane incidence theorem of the first author \cite{Ru}. Stevens and de Zeeuw \cite{SdZ} derived from it a point--line incidence theorem, which has enabled new applications to sum--product  type estimates, in spirit similar to those over the reals, based on the Szemer\'edi--Trotter theorem, see e.g. \cite{57}.

It was shown in \cite{RRS} that 
\begin{equation}\label{threshold} |A\pm A| + |AA| \gg |A|^{6/5} , \ \ \ A \subset \mathbb{F}_p, \ \ |A| \leq p^{5/8}.
\end{equation}
Shakan and Shkredov \cite{Shakan_S}, succeeded in improving the \eqref{threshold} to $\frac{6}{5} + c,$ for a certain $c >0$. Chen, Kerr and Mohammadi \cite{CKM} have recently achieved quantitative improvements to the value of $c$ in \cite{Shakan_S} by largely following its proofs, wherewithin they identified a more optimal way of applying incidence bounds.  

Today, after much effort, it appears unlikely that (but for a few exceptions) even weaker versions of Conjecture \ref{esconj}, the central one being the {\em weak Erd\H os--Szemer\'edi conjecture}, discussed in some detail in the real setting in \cite{MRSS}, can be fully resolved using the available incidence technology. However,  the question how far partial results based thereon can be pushed appears to be, at least on a certain level, interesting. To this effect, the third author and collaborators (see e.g., \cite{SS}, \cite{S_ineq}-\cite{Shkr_H+L}) developed a framework of methods, based on linear algebra and combinatorics, which have enabled a steady supply of improvements of the state of the art of sum-product theory. A recent paper \cite{MRSS} claims to have taken advantage of the latter techniques, over the reals, in what may be the best possible way.

In a loose sense, this paper attempts to establish a positive characteristic analog of some results in \cite{MRSS}. In particular,
Theorem~\ref{t:sd} gives a further improvement of the sum-product inequality, relative to that in \cite{CKM}, replacing the original proof  in \cite{Shakan_S} by an essentially different one. We do not expect to have our sum-product inequality improved further, within the reach of today's methodology. (Admittedly,  there are many instances when prognoses along the lines of the latter statement turn out to be false. If so, one can say in retrospect, they were stimulating.)


In addition to the standard sum--product problem, we present Theorem~\ref{t:fpmd} and its implication Theorem~\ref{t:AA-AA}, which are ``threshold--breaking'' in a slightly different sense. Theorem  \ref{t:fpmd}, or heuristically {\em few products imply many differences}, appears to be in interesting development, at least in the sense that currently available techniques, in fact, allow for it, apropos of the  weak Erd\H os--Szemer\'edi conjecture. The statement of Theorem \ref{t:AA-AA} can be viewed as a particular case of an Erd\H os--type geometric question about distinct values of bilinear forms on a plane point set, studied in the real setting in, e.g. \cite{IRR}.

We next present the three main inequalities, established here. Say, if $F=\F_p$, these inequalities clearly cannot hold for sets $A$, comparable in size with $p$.  The proofs of these inequalities rely on incidence results stemming from the point--plane incidence theorem from \cite{Ru} which in positive characteristic $p$ is constrained in terms of $p.$ It would be highly desirable to have some sort of a generalisation for a finite field $\F_q$, $q$ being a power of $p$, with a constraint expressed in terms of $q$ but there is no such a generalisation for now.

Since within each proof herein incidence results are used several times, the constraints may look at the first glance {\em ad hoc}, and one may be tempted to say ``for $|A|$ sufficiently small in terms of $p$'' instead.

We use the standard Vinogradov notations $\ll,\gg$ to hide absolute constants in inequalities, $\approx$ means both $\ll$ and $\gg$, and the symbols $\lesssim,\gtrsim$ suppress, in addition to constants, powers of $\log|A|$.

\begin{theorem}[Sum--product]\label{t:sd} Let $|AA|=M|A|,\;|A\pm A|=K|A|,$ for $|A|\subset F^*$. If $|A|<p^{18/35}$, then 
$$\max(K,M)\gtrsim|A|^{2/9}\,.$$ Moreover, if $K^3M|A|^3 < p^2$, then 
\begin{equation}\label{ineq:dif}K^4M^5\gtrsim |A|^2\,.\end{equation}
\end{theorem}

We remark that using the point--plane incidence bound, one can show \cite[Equation 6]{SdZ} that  $$|AA| \ll |A| \ \implies \ |A+A| \gg |A|^{3/2}.$$ Our next result surpasses this barrier and implies that $|AA| \ll |A|$, then $|A - A| \gg |A|^{3/2 + 1/24}$. 

\begin{theorem}[Few products, many differences] \label{t:fpmd} Let $|AA|=M|A|,\;|A-A|=K|A|,$ for $|A|\subset F^*$. If $M^2K^2|A|^3<p^2$ or $|A|<p^{24/49}$, then
\begin{equation}\label{ineq:frmd}K^{24}M^{36} \gtrsim |A|^{13}\,.\end{equation}\end{theorem}

The estimate of Theorem~\ref{t:fpmd} is only better than Theorem~\ref{t:sd} for small $M$. It would be interesting to obtain a similar estimate if $K$ pertained to $A+A$, rather than $A-A$ and even more interesting if a corresponding threshold--breaking statement in the vein of {\em few products imply many sums} could be established apropos of additive energy of $A$, see \cite{MRSS} for the real setting.

By following the proofs, it is easy to see that the product set $AA$ can be replaced by the ratio set $A/A$.

\medskip


 
It was proved in \cite[Corollary 15]{Ru}, \cite[Corollary 4]{RRS} that 
\begin{equation}\label{aa-aa} |AA-AA| \gg |A|^{3/2}, \ \ \ |A| \leq p^{2/3}.\end{equation}
Theorem~\ref{t:fpmd} enables one to  improve upon \eqref{aa-aa}. This can also be viewed as the special case of the general open question, concerning the minimum cardinality of set of values of the symplectic form on pairs of points in  a given set in  the plane $F^2$ (here the set being $A\times A$), see \cite[Theorem 4]{57} for a general geometric bound. We formulate the next theorem in slightly more generality.
 
 \begin{theorem}[Expansion] \label{t:AA-AA}
 		Let $A,B,C \subseteq  F^*$ be sets of approximately the same size  $|A| < p^{4/9}$ and $B\cap C\neq \emptyset$.  
 	Then for some positive $c>0$ one has 
 	\[
 	|AB-AC| \gtrsim |A|^{3/2+c} \,.
 	\] 
 	One can take any $c=1/96$ and  $c=1/56$ if $B=C$.  
 \end{theorem}

The powers of $\log|A|$ hidden in the $\gtrsim$ symbols can be easily tracked down, however they are not our concern.

Progress, achieved in this paper, is primarily due to further development of methodology founded by the third author, which enables a close to optimal multiple applications of incidence results (this was initiated in \cite{Shakan_S, CKM}). In particular, this calls for the use of several different energies, or moments of convolution, formally introduced in the next section. Of special importance here is the fourth additive energy $\E_4(A)$, owing to the forthcoming Corollary \ref{c:e4} of the Stevens--de--Zeeuw incidence theorem; in the Euclidean setting the same role was played by  the third moment $\E_3(A)$, owing to the Szemer\'edi--Trotter theorem. See, in particular, \cite{S_ineq, s_mixed, SS, sh_as} as well as \cite{Shakan}  for the general description of the approach, the closely related spectral method, and various applications in the context of the sum--product phenomenon.

\section{Preliminaries}

Let $A,B\subseteq F$ be some finite sets.
We  use representation function notations like $r_{A-B} (x)$, which counts the number of ways $x \in F$ can be expressed as a  difference $a-b$ with $a\in A$, $b\in B$, respectively. 

For a real $n>1$ we define the $n$th moment of the representation function, or energy (see \cite{SS}) as 
$$
\E_n (A,B) = \sum_x r^n_{A-B} (x)\,,
$$
writing just $\E_n(A)$ when $A=B$.

Owing to the fact that the equation 
$$
a-b = a'-b':\qquad a,a'\in A,\;b,b'\in B
$$
can be rearranged, one has as well that
$$
\E_2 (A,B) = \sum_x r^2_{A+B} (x)\,.
$$
If $n\geq 2$ is integer, then after resummation one has
$$
\E_{n}(A) = \sum_{x_1,\ldots,x_{n-1}} |A\cap (A+x_1)\cap \ldots\cap (A+x_{n-1})|^2.
$$
This means that if one partitions the set $A^n$ of $n$-tuples $(a_1,\ldots,a_n)$ into equivalence classes by translation, then $\E_n$ is the sum, over equivalence classes $[a_1,\ldots,a_n]$ of squares of the numbers of $n$-tuples in an equivalence class.

\bigskip 
Next we formulate incidence results to be used, in the form most adapted to our purposes. The first one is an adaptation of the first author's point-pane theorem \cite{Ru}.
\begin{theorem} \label{t:pp}
	Let $A,B,C \subset F$, with $\max(|A|,|B|,|C|)<\sqrt{|A||B||C|}<p$. Then
$$
|\{(a,b,c,a',b',c'):\,a,a'\in A,\,b,b'\in B,\,c,c'\in C \mbox{ and } a+bc = a'+b'c'\}| \ll (|A||B||C|)^{3/2}.
$$
\end{theorem}

The second one is a derived statement for point--line incidences due to Stevens and de Zeeuw~\cite{SdZ}.
\begin{theorem} \label{t:pl}
	Let $A, B,C,D \subset F$, with $|A||B||D|<p^2$ or $|A||C||D|<p^2$. Then
$$\begin{aligned}
|\{(a,b,c,d)\in A\times B\times C\times D:\; c= ab+ d\}| & \ll \min[ \, ( |A||B||C|)^{3/4}|D|^{1/2},\,(|A||C||D|)^{3/4}|B|^{1/2} \,] \\ & + |A||D| + |B||C|. \end{aligned}
$$
\end{theorem}
As to the forthcoming applications of Theorem \ref{t:pl}, we refer to the first term in its estimate as the {\em main term} and the remaining two as {\em trivial terms}, which in meaningful applications will be dominated by the main term.

\begin{corollary} \label{c:e4} Let $A\subseteq F^*$, $D\subseteq F$ with $|AA| \leq M|A|$ and $M|A|^2|D|<p^2$. Then 
\begin{equation}\label{e4bound}
\E_4(A,D):=\sum_x r^4_{A-D}(x) \ll \min( M^3|A|^2|D|^2,\, M^2|A||D|^3)\log|A|\,.
\end{equation}
Hence, the number of distinct equivalence classes $[a,b,c,d]$  of quadruples $(a,b,c,d)\in A^4$ by translation is $\gtrsim M^{-2} |A|^4$.
\end{corollary}
\begin{proof}
For $1\leq k \leq\min(|A|,|D|),$ let 
$$n_k:=| X_k := \{ x\in A-D:\,r_{A-D}(x)\geq k\} |\,.$$
We claim that
\begin{equation}\label{e:cl}
n_k \ll \min( M^3|A|^2|D|^2,\, M^2|A||D|^3) /k^4+M|D| /k \ll \min( M^3|A|^2|D|^2 /k^4,\, M^2|A||D|^3) /k^4\,,
\end{equation}
the term $M|D| /k $ getting subsumed owing merely to the above range of $k$.

Estimate \eqref{e4bound} then follows after dyadic summation in $k$. To justify \eqref{e:cl}, for each $x\in X_k$ there are $\geq k$ solutions to the equation $x= \alpha -d$, with $\alpha \in A,\,d\in D$. This means, there are $\geq k|A| n_k$ solutions to the equation 
$$
x= (\alpha/a) a - d  = ab - d: \qquad a \in A^{-1}, b\in AA,\,x\in X_k.
$$
Estimate \eqref{e:cl} follows after comparing the above lower bound with the upper bound furnished by Theorem \ref{t:pl} and rearranging. 

If we set $D=A$, the number $N$, of equivalence classes $[a,b,c,d]$ of quadruples  $(a,b,c,d) \in A^4$ by translation satisfies, by the Cauchy--Schwarz inequality, the lower bound
\begin{equation} \label{e:eq} N\geq |A|^8/\E_4(A) \gtrsim M^{-2} |A|^4.\end{equation}

\hfill $\Box$
\end{proof}

\section{Proof of Theorem \ref{t:sd}}
The presented proofs involving the sum and difference sets are somewhat different, the difference set case being easier.  We therefore present them separately, beginning with the difference set, despite the proof involving the sumset applies to the different set as well, in essence by replacing the truism \eqref{ones} therein with \eqref{one} below.

\medskip
\begin{proof} [Difference--product inequality]
Let $P\subseteq A-A$ be a set of popular differences, defined as follows: for  every $x\in P,\,r_{A-A}(x) \geq \frac{|A|}{2K}.$ The notions of popularity, as well as the accompanying notations $P,\Delta,$ etc. vary from one proof to another.

We further say that $P$ is popular {\em by mass}, meaning that, by the pigeonhole principle,
$$
|\{(a_1,a_2)\in A\times A;\,a_1-a_2\in P\}| \gg |A|^2.
$$
Consider the equation
\begin{equation}\label{one}
c-b=(a-b)-(a-c) = (d-b)-(d-c).
\end{equation}
Suppose, $x:=a-b$ and $y:=d-b$ are in  $P$.  Besides, $u:=a-c$, $v:=d-c$ are both in $A-A$. By definition of $P$, equation \eqref{one} has $\gg |A|^4$ solutions $(a,b,c,d)$. 

\renewcommand{\thefootnote}{\arabic{footnote}}

Clearly, equation \eqref{one} is translation-invariant, and an equivalence class $[a,b,c,d]$  by translation is fixed by the values of three of the five variables defined above, namely $x,y,u,v$, as well as $w:=c-b$. Each equivalence class provides a distinct solution of the system of equations
$$x,y \in P,\, u, v,w\in A-A:\, x-u = y - v = w.$$
It follows by the Cauchy--Schwarz inequality and Corollary \ref{c:e4} that
\begin{equation}\label{two}
|A|^4\lesssim \sqrt{\E_4(A)} \sqrt{|\{x,y \in P; u,v\in A-A:\, x-u = y - v \in A-A\}|} \;:= \;M|A|^2 \sqrt{X}.
\end{equation}
To bound the quantity $X$, we use popularity of the differences $x,y$ and dyadic localisation. Namely, for some $\Delta\geq 1$ and some $D\subseteq A-(A-A)$ one has
\begin{equation}\label{inter}\begin{aligned}
X & \ll \frac{K^2}{|A|^2} |\{a_1,a_2,a_3,a_4\in A; \;u,v\in A-A:\, a_1-(a_3-u) = a_2-(a_4 - v) \in A-A\}| \\ & \lesssim
\frac{K^2}{|A|^2}\Delta^2 |\{a_1,a_2\in A,d_1,d_2\in D\subseteq A-(A-A):\, a_1-d_1 = a_2-d_2 \in A-A\}| \\
& \leq  \frac{K^2}{|A|^2}\Delta^2  \sum_{w \in A-A} r_{A-D}(w)^2 \leq \frac{K^{5/2}}{|A|^{3/2}} \Delta^2 \left( \sum_{w} r_{A-D}(w)^4 \right)^{1/2}=\;\frac{K^{5/2}}{|A|^{3/2}} \sqrt{\E_4(A,D)}\,,\end{aligned}
\end{equation}
where the last inequality is an application of Cauchy-Schwarz\footnote{Here, as well as further in \eqref{upper} it is possible, on the technical level, to estimate  $\sum_{w \in A-A} r_{A-D}(w)^2$ in a slightly different way along the lines of the (fairly standard) argument presented between estimates  \eqref{six} and \eqref{e:e} in the forthcoming proof of Theorem \ref{t:fpmd}. Although that would save a factor $\log|A|$, contributed by $\E_4(A,D)$, we chose to do it here in a more streamlined way via Corollary \ref{c:e4}.}.

The  above ``popular'' set $D \subseteq A-(A-A)$ is defined via $r_{A+(A-A)}(d)\approx \Delta,\,\forall d\in D$.
 (The brackets in the subscript of the notation $r_{A+(A-A)}(d)$ mean that this is the number of representations of $d$ as the sum $d=a+x$, with $a\in A$ and $x\in A-A$, rather than $x=a+a'+a''$, with $a,a',a''\in A$.)

We now apply Corollary \ref{c:e4}, whose constraint in terms of $p$ will be satisfied either under assumption $K,M<|A|^{2/9}$ or by the assumption $K^3M|A|^3<p^2$, owing both cases to  the Pl\"unnecke's inequality $A-(A-A)\leq K^3|A|$ (see e.g. \cite[Section 6.5]{TV}).



 This enables one to continue the series of estimates \eqref{inter} as
\begin{equation}\label{three}\begin{aligned}
X & \lesssim  M^{3/2} \frac{K^{5/2}}{|A|^{1/2}} M^{3/2} (|D|\Delta^2)  \\
& \leq M^{3/2} \frac{K^{5/2}}{|A|^{1/2}} M^{3/2} E^+(A,A-A) \ll K^{4} M^{3} |A|^2\,,\end{aligned}
\end{equation}
where the last estimate has invoked Theorem \ref{t:pp}. Namely 
$$\begin{aligned}
\E_2(A,A-A) & =|\{(a_1,a_2,d_1,d_2)\in A^2\times (A-A)^2:\, a_1-d_1 =a_2-d_2\}| \\ 
& \leq |A|^{-2} |\{(a,a',d_1,d_2,b_1,b_2)\in A^2\times (A-A)^2 \times AA^2:\, b_1/a -d_1 =b_2/a'-d_2\}| \\
& \ll M^{3/2}K^{3/2}|A|^{5/2}\,.
\end{aligned}
$$
Checking that conditions of Theorem \ref{t:pp} have been satisfied by the assumptions on $|A|,K,M$ is straightforward, for the converse of inequality \eqref{ineq:dif} implies $KM<  |A|^{1/2}.$

Putting it together yields
$$
K^{4} M^{5} \gtrsim |A|^2\;\;\;\Rightarrow\;\;\;\max(K,M)\gtrsim |A|^{2/9}\,,
$$
concluding the proof of the difference-product inequality of  Theorem \ref{t:sd}.

\hfill $\Box$
\end{proof}

\begin{proof}[Sum--product inequality]
Let $|AA| \leq M|A|,\;|A+A| \leq K|A|$. We write the input conditions as inequalities, for further we will pass to a large subset of $A$.

Let $P$ be a set of popular sums, defined as follows.

\begin{equation}
P=P(A):=\left\{x\in A+A:\,r_{A+A}(x) \geq \epsilon \frac{|A|}{K}\right\},\label{pops}\end{equation}
for a small $\epsilon >0$, to be later chosen as $\sim \log^{-1}|A|$. This choice of the popular set is to be justified shortly.

By the pigeonhole principle
$$|\{(a,a') \in A\times A:\,a+a'\in P\}|\geq (1-\epsilon)|A|^2\,.$$

Furthermore,  let $A'\subseteq A$ be 
\begin{equation}\label{popabs}
A'=A'(A):=\{a'\in A: |\{a'' \in A: a'+a'' \in P(A)\}| \geq (1-\epsilon)|A|\}\,,
\end{equation}
so $|A'|\geq (1-\epsilon) |A|$.

Let $P'\subseteq A'-A'$ be popular by energy $\E_{4/3}(A')$. Namely $x\in P'$ if for some $\Delta'\geq 1$,  $\Delta'\leq r_{A'-A'}(x) <2\Delta'$, and 
and
$$
\E_{4/3}(A') :=\sum_{x\in A'-A'}r_{A'-A'}(x)^{4/3} \gtrsim |P'|{\Delta'}^{4/3}.
$$
The reason why we deal with the additive energy $\E_{4/3}(A')$ will be clear from the sequel, as well as the raison d'\^etre of the following lemma.

\begin{lemma}\label{subs} There exists  $B\subseteq A$,  with $|B|\gg |A|$, such that $\E_{4/3}(B'(B)) \gg \E_{4/3}(B)$, 
where $B'\subseteq B$ is defined relative to $B$ replacing $A$ in conditions \eqref{pops}, \eqref{popabs}.\end{lemma}
\begin{proof} 
Indeed, suppose for contradiction that, say $\E_{4/3}(A'(A)) < \E_{4/3}(A)/10,$ i.e. at least $90$ per cent of the energy is supported on a thin subset $A\setminus A'$, of cardinality $|A\setminus A'| <  \epsilon|A|.$ Throw away the latter subset from $A$, redefine what remains as $A$, with $A'$ being redefined accordingly via \eqref{popabs}, and attempt to repeat the procedure some $\epsilon^{-1}$ times. If this was possible, then in the end of it one is left with a subset $A_\epsilon$ of $A$ of cardinality $|A_\epsilon|\geq (1-\epsilon)^{\epsilon^{-1}}|A| \gg |A|$, with $\E_{4/3}(A_\epsilon) < 10^{-\epsilon^{-1}}\E_{4/3}(A)$. Choosing $\epsilon = \log^{-1}|A|$ is clearly a contradiction, for trivially $\E_{4/3}(A_\epsilon)\geq |A_\epsilon|^2\gg|A|^2.$

\hfill $\Box$\end{proof}

For the rest of the proof of the sum-product estimate, without loss of generality we  take $B=A$, in other words 
assuming that  
\begin{equation}\label{crass}\E_{4/3}(A')\gg\E_{4/3}(A)\,,\end{equation}
to be used in the end of the proof. We also set $\epsilon=\log^{-1}|A|$ in  \eqref{pops}, \eqref{popabs}.

Consider now a variant of equation \eqref{one} as follows:
\begin{equation}\label{ones}
-c+b=(a+b)-(a+c) = (d+b)-(d+c)\,.
\end{equation}
Let us make popularity assumptions as to the variables $a,b,c,d$. By definition of the sets $A'$  and $P'$, it follows that the number of solutions $:=\sigma$ of equation \eqref{ones}, when the difference $b-c\in P'$ and {\em all }the four sums $x:=a+b$, $y:=a+c$, $u:=d+c$ and $v:=d+b$ 
involved are in $P$ is bounded from below as
\begin{equation}\label{lower}
\sigma \geq (1-4\epsilon)|P'|\Delta|A|^2.\end{equation}

Next we obtain the upper bound for the number of solutions $(a,b,c,d)$ of equation \eqref{ones} under the constraints above.  Equation \eqref{ones} is invariant to a simultaneous shift of $b,c$ by some $t$ and $a,d$ simultaneously by $-t$. We say $[a,b,c,d]$ is equivalent to $[a',b',c', d']$ if $$(a,b,c,d) = (a', b' , c' , d') + (t,-t,-t,t).$$ Each equivalence class $[a,b,c,d]$ yields a different solution of the system of equations
$$x,y,u,v\in P,\,w\in P':\,x-y = v-u=w.$$
If $r([a,b,c,d])$ denotes the number of quadruples $(a,b,c,d)$ in an equivalence class, then
$$
\sum_{[a,d,c,d]} r^2([a,b,c,d]) = \sum_{x\in A-A} r^2_{A-A}(x)r^2_{A-A}(-x) = \E_4(A).
$$
An upper bound for the quantity $\sigma$ -- similar to estimate \eqref{two} -- now follows by the Cauchy--Schwarz inequality. Invoking also the lower bound \eqref{lower} yields

\begin{equation}\label{here}
 |A|^2|P'|\Delta' \lesssim \sqrt{\E_4(A)} \sqrt{| \{x,y,u,v\in P,\,w\in P':\,x-y = v-u=w\}|}\,.
\end{equation}
Popularity of the sums $x,y,u,v$  together with  Corollary \ref{c:e4} to bound $\E_4(A)$ yield
$$
 |A|^2|P'|\Delta'  \lesssim M  K^2 \sqrt{|\{a_1,\ldots,a_8\in A:\, a_1+a_2-a_3-a_4=a_5+a_6-a_7-a_8\in P'\}|}\,.
$$
We proceed similar to estimates \eqref{inter}: there exists a popular subset $D\subseteq A+A-A$ where $\forall d\in D,\,r_{A+A-A}(d)\approx\Delta,$ for some $\Delta\geq 1$ (here, contrary to the difference set case $r_{A+A-A}(d)$ means the number of representations $d=a+a'-a''$, with $a,a',a''\in A$), such that one gets

\begin{equation}\label{upper}\begin{aligned}
 |A|^2|P'|\Delta'   &\lesssim
M K^2 \Delta \sqrt{|\{a_1,a_2\in A,\;d_1,d_2\in D\subseteq A+A-A:\, a_1-d_1 = a_2-d_2 \in P'\}|} \\
& \leq M K^2 \Delta |P'|^{1/4} \E_4^{1/4}(A,D)\,, \end{aligned}
\end{equation}
after another use of the Cauchy-Schwarz inequality.

Using Corollary \ref{c:e4} to estimate  $\E_4(A,D)$ -- its applicability in terms of the constraints in terms of $p$ being the same as it was in the difference set case, see the argument following  \eqref{inter} -- we conclude that
\begin{equation}\label{interpups}
 |P'|^{3/4}\Delta' \lesssim M^{7/4} K^2 |A|^{-3/2} (|D|\Delta^2)^{1/2} \leq M^{7/4} K^2   |A|^{-3/2}  \sqrt{\T_3(A)},
\end{equation}
where 
$$\T_3(A):=|\{(a_1,\ldots,a_3')\in A^6:\,
a_1+a_2+a_3 = a_1'+a_2'+a_3'\}|\,.
$$ 
The quantity $\T_3(A)$ can be bounded as follows. 
 One can localise $a_2-a_3=x,\, a_2'-a_3'=x'$ to some popular set $D_1\subset A-A$ with  $r_{A-A}(x)\approx \Delta_1, \,\forall x\in D_1$  and apply Theorem \ref{t:pp}, so
\begin{equation}\label{five}
\T_3(A) \lesssim M^{3/2}|A| (|D_1|^{3/2}\Delta_1^2) \leq M^{3/2}|A| ( \E_{4/3}(A) )^{3/2}\,.
\end{equation}
It is easy to verify that the assumptions on $|A|,K,M$ ensure that the conditions of Theorem \ref{t:pp} have been amply satisfied.

It follows by definition of the popular set $P'$ after substituting  bound \eqref{five}  into \eqref{interpups} that

$$
(\E_{4/3}(A'))^{3/4}  \lesssim |P'|^{3/4} \Delta' \lesssim M^{5/2}K^2|A|^{-1} (\E_{4/3}(A))^{3/4}.
$$
Hence, by \eqref{crass}, one cancel $\E_{4/3}(A')\gg \E_{4/3}(A)$ and be left with $$
|A|\lesssim  M^{5/2}K^2\,,
$$
which proves Theorem \ref{t:sd}.

\hfill $\Box$
\end{proof}

\section{Proof of Theorem \ref{t:fpmd}}

Return to  relations \eqref{one}, \eqref{two}, with the notations $x,y,u,v,w$ as they were introduced apropos of \eqref{one}, \eqref{two}, and observe that $u-v=a-d:=z\in A-A$. Suppose that $z$ is popular  my mass (i.e with  say $r_{A-A}(z)\geq \frac{|A|}{10K}$) and so are $x$ and $y$, set $P\subseteq A-A$ in this section again denote the set of such popular differences.

From \eqref{two} we have
$$
u = x - w,\qquad v = y - w \qquad \Leftrightarrow \qquad u,v  \;\in \;(A-A) \cap (P -w)\,:=\P_w.
$$
We can now rewrite \eqref{two} as
\begin{equation}\label{six}
|A|^4\lesssim M|A|^2 \sqrt{ \sum_{w\in A-A} |\{ u,v\in P_w:\,u-v\in P\}| } :=M|A|^2\sqrt{X}.
\end{equation}
Let us estimate $|P_w|$, sorting $A-A=\{w_1,\ldots,w_{A-A}\}$ in non-decreasing order by the value of $r_{P-(A-A)}(w)$.

Set 
$$
n_k:= |W_k:=\{w \in P-(A-A):\;r_{P-(A-A)}(w)\geq k\}|.
$$
This means, for every $w\in W_k$ the equation $w=x-u:\,x\in P,\,u\in A-A$ has $\geq k$ solutions. Hence, the equation
$$w=t/a-u:\,w\in W_k,\,t\in AP,\,u\in A-A,\,a\in A$$ has $\geq k|A|n_k$ solutions. Furthermore, $AP\subseteq AA-AA$, and $\forall\,t\in AP,\,r_{AA-AA}(t)\gg |A|/K$. It follows that 
\begin{equation}\label{f:AP}
	|AP|\ll M^2K|A| \,.
\end{equation}

Apply Theorem \ref{t:pl} to get the upper bound for the number of solutions of the latter equation. Note that the $p$-condition of Theorem \ref{t:pl} becomes $p^2> M^2K^2|A|^3$, which is satisfied, in particular, for if $|A|<p^{24/49}$, when assuming $K^{24}M^{36}<|A|^{13}$ (or there is nothing to prove) implies that $M^2K^2<|A|^{13/12}.$ 

Hence, one concludes that 
$$
k|A|n_k \ll M(K|A|)^{1/2} (|A|n_k)^{3/4} (K|A|)^{3/4} + M^2K^2|A|^2\,.
$$
Rearranging, dropping the second term since it follows by definition of $n_k$ that $k\leq K|A|$,  yields
$$
n_k \ll M^4 K^5|A|^4/k^4\,.
$$ 
Inverting the latter bound yields 
$$
k_n \ll M K^{5/4}|A|n^{-1/4}\,,
$$
which means that for $w=w_n$ on the list, one has
\begin{equation}
|P_{w_n}|\ll \min( |A-A|,\, M K^{5/4}|A|n^{-1/4})\,.
\label{e:e}\end{equation}

Furthermore, given $w$, by another application of Theorem \ref{t:pl} (the $p$-condition check being the same as done above) one has
\begin{equation}\label{sev}\begin{aligned}
|\{ u,v\in P_w:\,u-v = z\in P\}| & \ll \frac{1}{|A|} \{ u,v\in P_w:\,u-v = t/a, \mbox{ with } a\in A,\, t\in AP\}| \\
&\ll |A|^{-1} ( |P_w|^{3/2} (M^2 K|A|)^{1/2} |A|^{3/4} + M^2 K|A||P_w|) \\
& \ll MK^{1/2}|A|^{1/4}   |P_w|^{3/2} + M^2 K^2|A|\,,\end{aligned}
\end{equation}
where in the last term the trivial bound $ |P_w|\leq K|A|$ has been used.

It follows from  \eqref{e:e} that
$$
\sum_{w\in A-A}|P_w|^{3/2} \ll M^{3/2} K^{15/8}|A|^{3/2} \sum_{n=1}^{|A-A|} n^{-3/8} \ll M^{3/2} K^{5/2}|A|^{17/8}\,.
$$
Substituting the latter bound into \eqref{sev} one sees that the quantity $X$ introduced in  \eqref{six} obeys
\begin{equation}\label{sev'}
X\ll M^{5/2} K^{3}|A|^{19/8}  + M^2 K^3|A|^2 \ll M^{5/2} K^{3}|A|^{19/8}\,.
\end{equation}
It follows from \eqref{six} that
$$
|A|^{13} \lesssim M^{36} K^{24}\,,
$$
as claimed by Theorem \ref{t:fpmd}.

\hfill $\Box$

\section{Proof of Theorem \ref{t:AA-AA}}

\begin{proof}
	We give two approaches, the first one allowing for better quantitative estimates, the second one being more general. It is easy to check that the proof holds if, e.g., $|A|<p^{4/9}$. 
	
	Set $s=|AB-AC|$, $M=|AB|$.
	Applying Theorem \ref{t:pp}, one has, for $M|A|^3<p^2$ (see details in \cite[Corollary 4]{RRS}) that 
\[
	|AB-AC| \gg  M^{1/2} |A|^{3/2} \,.
\]
	Otherwise, since there $B\cap C\neq \emptyset$,  one has $|AB-AC|\geq |A-A|=K|A|$. The proof of Theorem \ref{t:fpmd} allows for replacing the product set $AA$ with $AB$, with $|B|\approx|A|$, the same concerning inequality \eqref{ineq:frmd}. I.e., with $|AB|=M|A|$, one has
\[
	|A|^{13} \lesssim M^{36} K^{24} \le (s^2/|A|^3)^{36} (s/|A|)^{24} = s^{96} |A|^{-132}
\]
	or, in other words, $s\gtrsim |A|^{145/96}$. 
	In the special case $B=C$, we can estimate size of $|AP|$ in \eqref{f:AP} as $s$. 
	It gives us $n_k \ll s^2 |A|^2 K^3/k^4$, further 
	the main term in  \eqref{sev'} is $K^{7/4} |A|^{9/8} s^{5/4}$.  
	Thus the second term in \eqref{sev} is negligible again.
	Hence 
\[
	|A|^4 \lesssim M^2 K^{7/4} |A|^{9/8} s^{5/4} \le (s^2 /|A|^3)^2 (s/|A|)^{7/4} |A|^{9/8} s^{5/4} = s^{7} |A|^{-53/8}
\]
	or, in other words, $s\gtrsim |A|^{85/56}$.

	\medskip
	Alternatively, we present an argument, which uses the Balog--Szemer\'{e}di--Gowers Theorem, see e.g.  \cite[Section 2.5]{TV}, which has more potential for generalisation. Set  $\sigma := \sum_x r^2_{AB-AC} (x)$.  
	By the Cauchy-Schwarz inequality
	\[
	|A|^8 \le s \sigma \,.
	\]   
	Write $s = N|A|^{3/2}$, for some $N$. Then $\sigma \ge |A|^{13/2} / N$.  
	By \cite[Theorem 32, Remark 33]{sh_as}, one has the following estimate, provided that  $K|A|^3<p$:
	\[
	\sigma  \lesssim |A|^5 (\E_2^\times (A, B) )^{1/2} \,.
	\]
	where $\E^\times_2$ is multiplicative energy, defined in the standard way. Thus $\E_2^\times (A) \gtrsim |A|^3 / N^2$. 
	
	Using the Balog--Szemer\'{e}di--Gowers Theorem \cite{TV}, one finds $A'\subseteq A$ such that $|A'| \gtrsim_N |A|$ and $|A'A'| \lesssim_N |A'|$, the symbols   $\gtrsim_N, \,\lesssim_N$ absorbing universal powers of $N$.
	
	Applying Theorem \ref{t:sd} to the set $A'$ (it's easy to see that its conditions are satisfied) yields
	\[
	|AB-AC| \ge |A'-A'| \gtrsim |A'|^{37/24} \gtrsim_N |A|^{37/24}\,,
	\]
	as required. 
	\hfill $\Box$
\end{proof}

\noindent{Misha Rudnev\\
	School of Mathematics,\\
	University Walk, Bristol BS8 1TW, UK\\
	{\tt misarudnev@gmail.com}

\bigskip
\noindent{George Shakan\\
Department of Mathematics \\
University of Illinois \\
Urbana, IL 61801, U.S.A.\\
{\tt shakan2@illinois.edu}

\bigskip
\noindent{Ilya D. Shkredov\\
	Steklov Mathematical Institute,\\
	ul. Gubkina, 8, Moscow, Russia, 119991}
\\
and
\\
IITP RAS,  \\
Bolshoy Karetny per. 19, Moscow, Russia, 127994\\
and 
\\
MIPT, \\ 
Institutskii per. 9, Dolgoprudnii, Russia, 141701\\
{\tt ilya.shkredov@gmail.com}

\end{document}